\documentclass[letterpaper]{amsart}
\usepackage[english]{babel}
\usepackage{amsmath}
\usepackage{amssymb}
\usepackage{amsfonts}
\usepackage{amsthm}
\usepackage{mathtools}
\usepackage{pstricks}
\usepackage{pst-node}
\usepackage{pst-plot}
\usepackage{pst-3dplot}
\usepackage{booktabs}
\usepackage{multirow}
\usepackage{subfigure}
\usepackage{pdflscape}
\usepackage{url}
\usepackage[all,cmtip]{xy}
\allowdisplaybreaks[3]
\newcommand{\ds}{\displaystyle}
\renewcommand{\subset}{\subseteq}
\renewcommand{\supset}{\supseteq}

\newcommand{\overbar}[1]{\mkern 1.5mu\overline{\mkern-1.5mu#1\mkern-1.5mu}\mkern 1.5mu}
\newcommand{\R}{\ensuremath{\mathbf{R}}}

\newcommand{\Rplus}{\ensuremath{\mathbf{R}_{\geqslant0}}}

\newcommand{\Rplusnul}{\ensuremath{\mathbf{R}_{>0}}}

\newcommand{\Rn}{\ensuremath{\mathbf{R}^n}}
\newcommand{\Rplusn}{\ensuremath{\mathbf{R}_{\geqslant0}^n}}

\newcommand{\N}{\ensuremath{\mathbf{N}}}
\newcommand{\Nnul}{\ensuremath{\mathbf{N}\setminus\{0\}}}
\newcommand{\Nn}{\ensuremath{\mathbf{N}^n}}
\newcommand{\Z}{\ensuremath{\mathbf{Z}}}
\newcommand{\Zn}{\ensuremath{\mathbf{Z}^n}}

\newcommand{\Zplus}{\ensuremath{\mathbf{Z}_{\geqslant0}}}

\newcommand{\Zplusn}{\ensuremath{\mathbf{Z}_{\geqslant0}^n}}

\newcommand{\Zplusnul}{\ensuremath{\mathbf{Z}_{>0}}}

\newcommand{\Q}{\ensuremath{\mathbf{Q}}}

\newcommand{\C}{\ensuremath{\mathbf{C}}}

\newcommand{\Qp}{\ensuremath{\mathbf{Q}_p}}

\newcommand{\Qpn}{\ensuremath{\mathbf{Q}_p^n}}

\newcommand{\Zp}{\ensuremath{\mathbf{Z}_p}}
\newcommand{\Zpx}{\ensuremath{\mathbf{Z}_p^{\times}}}

\newcommand{\Zpn}{\ensuremath{\mathbf{Z}_p^n}}
\newcommand{\Zpxn}{\ensuremath{(\mathbf{Z}_p^{\times})^n}}
\newcommand{\F}{\ensuremath{\mathbf{F}}}
\newcommand{\Fp}{\ensuremath{\mathbf{F}_p}}
\newcommand{\Fpn}{\ensuremath{\mathbf{F}_p^n}}
\newcommand{\Fpcross}{\ensuremath{\mathbf{F}_p^{\times}}}
\newcommand{\Fpcrossn}{\ensuremath{(\mathbf{F}_p^{\times})^n}}

\renewcommand{\N}{\Zplus}
\renewcommand{\Nnul}{\Zplusnul}
\renewcommand{\Nn}{\Zplusn}
\renewcommand{\phi}{\varphi}

\DeclareMathOperator{\supp}{supp}
\DeclareMathOperator{\mult}{mult}
\DeclareMathOperator{\ord}{ord}

\newcommand{\I}{\ensuremath{\mathcal{I}}}

\newcommand{\Zf}{\ensuremath{Z_f}}

\newcommand{\Zfg}{\ensuremath{Z_{f,gdx}}}
\newcommand{\ZI}{\ensuremath{Z_{\mathcal{I}}}}
\newcommand{\ZIg}{\ensuremath{Z_{\mathcal{I},gdx}}}
\newcommand{\Zfs}{\ensuremath{Z_f(s)}}

\newcommand{\Zfgs}{\ensuremath{Z_{f,gdx}(s)}}
\newcommand{\ZIs}{\ensuremath{Z_{\mathcal{I}}(s)}}
\newcommand{\ZIgs}{\ensuremath{Z_{\mathcal{I},gdx}(s)}}

\newcommand{\ff}{\ensuremath{\mathbf{f}}}
\newcommand{\Zff}{\ensuremath{Z_{\ff}}}
\newcommand{\Zffg}{\ensuremath{Z_{\ff,gdx}}}
\newcommand{\Zffs}{\ensuremath{Z_{\ff}(s)}}
\newcommand{\Zffgs}{\ensuremath{Z_{\ff,gdx}(s)}}

\newcommand{\Gf}{\ensuremath{\Gamma_f}}
\newcommand{\Gff}{\ensuremath{\Gamma_{\ff}}}
\newcommand{\Gg}{\ensuremath{\Gamma_g}}
\newcommand{\GI}{\ensuremath{\Gamma_{\mathcal{I}}}}
\newcommand{\Gglf}{\ensuremath{\Gamma^{\mathrm{gl}}_f}}

\newcommand{\ft}{\ensuremath{f_{\tau}}}
\newcommand{\fft}{\ensuremath{\ff_{\tau}}}
\newcommand{\fbart}{\ensuremath{\overline{f_{\tau}}}}

\newcommand{\Df}{\ensuremath{\Delta_f}}
\newcommand{\Dff}{\ensuremath{\Delta_{\ff}}}
\newcommand{\Dg}{\ensuremath{\Delta_g}}
\newcommand{\DI}{\ensuremath{\Delta_{\mathcal{I}}}}
\newcommand{\Dfg}{\ensuremath{\Delta_{f,g}}}
\newcommand{\Dffg}{\ensuremath{\Delta_{\ff,g}}}
\newcommand{\DIg}{\ensuremath{\Delta_{\mathcal{I},g}}}
\newcommand{\Dft}{\ensuremath{\Delta_f(\tau)}}
\newcommand{\Dfft}{\ensuremath{\Delta_{\ff}(\tau)}}

\newcommand{\DIt}{\ensuremath{\Delta_{\mathcal{I}}(\tau)}}

\newcommand{\gd}{\ensuremath{g_{\delta}}}
\newcommand{\fd}{\ensuremath{f_{\delta}}}
\newcommand{\ffd}{\ensuremath{\ff_{\delta}}}
\newcommand{\gbard}{\ensuremath{\overline{g_{\delta}}}}
\newcommand{\fbard}{\ensuremath{\overline{f_{\delta}}}}
\newcommand{\ffbard}{\ensuremath{\overline{\ff_{\delta}}}}
\newcommand{\Ld}{\ensuremath{L_{\delta}}}
\newcommand{\Nd}{\ensuremath{N_{\delta}}}
\newcommand{\Pd}{\ensuremath{P_{\delta}}}
\newcommand{\Qd}{\ensuremath{Q_{\delta}}}
\newcommand{\Sd}{\ensuremath{S_{\delta}}}

\theoremstyle{plain}
\newtheorem{theorem}{Theorem}[section]
\newtheorem{lemma}[theorem]{Lemma}
\newtheorem{proposition}[theorem]{Proposition}
\newtheorem{corollary}[theorem]{Corollary}

\theoremstyle{definition}
\newtheorem{definition}[theorem]{Definition}

\newtheorem{example}[theorem]{Example}

\theoremstyle{remark}
\newtheorem{remark}[theorem]{Remark}

\newtheorem{notation}[theorem]{Notation}

\title[Igusa's local zeta function associated to a polynomial measure]{Igusa's $p$-adic local zeta function\\associated to a polynomial mapping\\and a polynomial integration measure}
\author{Bart Bories}
\address{Department of Mathematics, KU Leuven, Celestijnenlaan 200b -- box 2400, 3001 Leuven, Belgium}
\email{bart.bories@wis.kuleuven.be}
\date{\today}

\begin{document}
\begin{abstract}
For $p$ prime, we give an explicit formula for Igusa's local zeta function associated to a polynomial mapping $\ff=(f_1,\ldots,f_t):\Qpn\to\Qp^t$, with $f_1,\ldots,f_t\in\Zp[x_1,\ldots,x_n]$, and an integration measure on \Zpn\ of the form $|g(x)||dx|$, with $g$ another polynomial in $\Zp[x_1,\ldots,x_n]$. We treat the special cases of a single polynomial and a monomial ideal separately. The formula is in terms of Newton polyhedra and will be valid for \ff\ and $g$ sufficiently non-degenerated over \Fp\ with respect to their Newton polyhedra. The formula is based on, and is a generalization of results in \cite{DH01}, \cite{HMY07}, and \cite{VZ08}.
\end{abstract}

\maketitle

\setcounter{tocdepth}{1}

\tableofcontents

\section*{Introduction}
\subsection*{Motivation}
Introduced by Weil \cite{Weil65} in 1965 and first studied by Igusa \cite{Igu74,IguFHD}, the $p$-adic local zeta function of Igusa is closely related to the numbers $N_l$ of solutions of the polynomial congruences $f(x)\equiv0\bmod p^l$ for $l\geqslant1$. For instance, Igusa \cite{Igu74} and Denef \cite{Den84} used the $p$-adic zeta function to prove the rationality of the series generated by the numbers $N_l$. In 1992, Denef and Loeser \cite{DL92} obtained the topological zeta function as a kind of limit of $p$-adic zeta functions; later they presented in \cite{DL98} the finer and intrinsically defined motivic zeta function. All these invariants were originally associated to one polynomial or analytic function in several variables over a certain field (a $p$-adic field, \C, and an arbitrary field of characteristic zero, respectively), but the concepts have been generalized in many ways during the research that followed. For instance, we now associate these zeta functions to several polynomials or to an ideal in a polynomial ring. See, e.g., \cite{HMY07,VZ08,VV08,VVmcid2}.

Nowadays one is also interested in zeta functions of a number of polynomials and a more general measure or differential form. For example, in \cite{NV10bis} Veys and N{\'e}methi use this notion to prove a remarkable result on the generalized Monodromy Conjecture. For other applications, see, e.g., \cite{ACLM02,ACLM05}. In this respect, it is useful to examine which known results and formulas have their analogue in this generalized context. In this paper, we adapt results from \cite{DH01,HMY07,VZ08} and prove an explicit formula for Igusa's $p$-adic zeta function of a strongly non-degenerated polynomial mapping and a \lq polynomial measure\rq. Although we state the formula for the $p$-adic one, the result can also be formulated for the topological and the motivic zeta function. In \cite{BorTVBN} we will use this formula to investigate the analogue of Veys and N{\'e}methi's result for a related, stronger conjecture of Igusa--Denef--Loeser.

\subsection*{Igusa's zeta function of a polynomial mapping and a polynomial measure}
For a prime $p$, we denote by \Qp\ the field of $p$-adic numbers and by \Zp\ its subring of $p$-adic integers. Denote by $|\cdot|$ the $p$-adic norm on \Qp. On \Qpn, $n\in\Nnul$, we consider the Haar measure, so normalized that \Zpn\ has measure one, and we denote it by $|dx|=|dx_1\wedge\cdots\wedge dx_n|$. For a measurable subset $A\subset\Qpn$, we denote its measure by $\mu(A)$. Let us start with the definition of Igusa's local zeta function.

\begin{definition}[Igusa's $p$-adic local zeta function] Let $p$ be a prime number and $f(x)=f(x_1,\ldots,x_n)$ a polynomial in $\Qp[x_1,\ldots,x_n]$. Igusa's ($p$-adic) local zeta function associated to $f$ is defined as
\begin{equation*}
\Zf:\{s\in\C\mid\Re(s)>0\}\to\C:s\mapsto\Zfs=\int_{\Zpn}|f(x)|^s|dx|.
\end{equation*}

More generally, one defines Igusa's local zeta function associated to a polynomial mapping $\ff=(f_1,\ldots,f_t):\Qpn\to\Qp^t$, with $f_1,\ldots,f_t\in\Qp[x_1,\ldots,x_n]$, as
\begin{equation}\label{def_igusa_ff}
\Zff:\{s\in\C\mid\Re(s)>0\}\to\C:s\mapsto\Zffs=\int_{\Zpn}\|\ff(x)\|^s|dx|,
\end{equation}
where $\|\ff(x)\|=\max_i|f_i(x)|$.

To an ideal \I\ of $\Zp[x_1,\ldots,x_n]$, we associate Igusa's local zeta function
\begin{equation*}
\ZI:\{s\in\C\mid\Re(s)>0\}\to\C:s\mapsto\ZIs=\int_{\Zpn}|\I(x)|^s|dx|,
\end{equation*}
where $|\I(x)|=\max\{|f(x)|\mid f\in \I\}$.

We can further generalize these notions by considering alternative, \lq polynomial\rq\ integration measures on \Zpn, i.e., measures of the form $|g(x)||dx|$, with $g$ another polynomial in $\Qp[x_1,\ldots,x_n]$. For $f$, \ff, and \I\ as above, we put
\begin{gather*}
\Zfgs=\int_{\Zpn}|f(x)|^s|g(x)||dx|,\qquad\Zffgs=\int_{\Zpn}\|\ff(x)\|^s|g(x)||dx|,\\
\text{and}\qquad\ZIgs=\int_{\Zpn}|\I(x)|^s|g(x)||dx|.
\end{gather*}
These are the zeta functions that we will study here.
\end{definition}

\begin{remark}
In general, we associate $p$-adic zeta functions to polynomial mappings $\ff=(f_1,\ldots,f_t):\Qpn\to\Qp^t$, rather than to ideals $(f_1,\ldots,f_t)\lhd\Qp[x_1,\ldots,x_n]$, since \Zff, as defined in \eqref{def_igusa_ff}, depends on the specific polynomials $f_1,\ldots,f_t$, not only on the ideal generated by them. However, if we consider polynomials $f_1,\ldots,f_t$ over the $p$-adic integers, it turns out that \Zff\ only depends on the ideal $(f_1,\ldots,f_t)\lhd\Zp[x_1,\ldots,x_n]$; i.e., two polynomial mappings over \Zp\ give rise to the same zeta function, if their composing polynomials generate the same ideal of $\Zp[x_1,\ldots,x_n]$. This yields a proper definition of Igusa's zeta function of a $\Zp[x_1,\ldots,x_n]$ ideal, which coincides with the one given above. Indeed, for every set $\{f_1,\ldots,f_t\}$ of generators of an ideal $\I\lhd\Zp[x_1,\ldots,x_n]$, we have that
\begin{equation*}
|\I(x)|=\max_{f\in\I}|f(x)|=\max_{1\leqslant i\leqslant t}|f_i(x)|
\end{equation*}
for every $x\in\Zpn$.

Following \cite{HMY07}, we will use the terminology of ideals when it comes to monomial ideals (Sections~\ref{prem}--\ref{formule1}), otherwise we will speak in terms of polynomial mappings.
\end{remark}

Using resolution of singularities, Igusa \cite{Igu74} proved in 1974 that his $p$-adic zeta function \Zf, associated to one polynomial $f$, is a rational function in the variable $t=p^{-s}$. An alternative proof (based on $p$-adic cell decomposition) of this important fact was obtained ten years later by Denef \cite{Den84}. With the same techniques it can be proved that this rationality result still holds for the other versions of Igusa's local zeta function, described above. Hence all of these functions allow a meromorphic continuation to the whole complex plane, that we denote with the same symbol.

Considering Newton polyhedra (see Definition~\ref{def_NP} below), Denef and Hoornaert obtained in \cite{DH01} a very explicit formula for Igusa's local zeta function of a single polynomial, for a special class of polynomials, namely those who are non-degenerated over \Fp\ with respect to all the faces of their Newton polyhedron (see Definition~\ref{def_non-degenerated}). In their paper \cite{HMY07} of 2007, Howald, Musta{\c{t}}{\u{a}}, and Yuen prove a similar formula for Igusa's local zeta function \ZI, associated to a monomial ideal $\I\lhd\Z[x_1,\ldots,x_n]$ and the usual integration measure. In 2008, Veys and Z{\'u}{\~n}iga-Galindo \cite{VZ08} generalized those two formulas to a formula for Igusa's local zeta function \Zff\ of a polynomial mapping \ff, that is strongly non-degenerated over \Fp\ with respect to its Newton polyhedron (Definition~\ref{def_strong_non-degenerated}).

Our goal is to adapt those formulas to cover the generalized versions \Zfg, \Zffg, and \ZIg\ of Igusa's local zeta function, associated to a polynomial measure $|g(x)||dx|$ on $\Zpn$, where $g$ is a polynomial in $\Zp[x_1,\ldots,x_n]$ that is also non-degenerated over $\F_p$ with respect to its Newton polyhedron, and where the pairs $(f,g)$ and $(\ff,g)$ satisfy a supplementary non-degeneracy condition, described in Definitions~\ref{non-degeneratedII} and \ref{non-degeneratedIII}, respectively.

\subsection*{Overview}
In the first section, we list all the needed definitions, notations and results that are already known, without much explanation. For more details on Igusa's zeta function and Newton polyhedra, we refer to \cite{DH01}, \cite{HMY07}, and \cite{VZ08}; for more background on convex geometry, we refer to \cite{Roc70}. In Section~\ref{formule1}, we derive a formula for Igusa's zeta function associated to a monomial ideal and a measure $|g(x)||dx|$. A formula for Igusa's local zeta function of a single polynomial $f$ and a measure $|g(x)||dx|$, is proved in Section~\ref{formule2}. Section~\ref{sectHensel} is a preparation for this section, here we derive a formula for the same integral as appearing in the definition of \Zfgs, but with integration domain \Zpxn\ instead of \Zpn. The proof of this formula is based on Hensel's Lemma. In the last section, we state the most general formula, i.e., for Igusa's zeta function of a polynomial mapping and a polynomial integration measure.

All definitions and results mentioned in this paper have straightforward analogues over any $p$-adic field, i.e., any field $K$ such that $[K:\Qp]<\infty$.

\section*{Acknowledgements}
The author would like to thank Wim Veys for proposing the problem and for many useful suggestions. He also wants to thank the referee for reviewing the paper and for some interesting comments.

\section{Preliminaries}\label{prem}
\begin{remark}
In what follows, we will consider polynomials over \Zp, because this is more convenient when dealing with the non-degeneracy conditions we use throughout the paper. Note that this does not affect the generality; indeed, every polynomial $f$ over \Qp\ can be written as $f=p^{-i}\tilde{f}$ for some $i\in\Zplus$ and some $\tilde{f}\in\Zp[x_1,\ldots,x_n]$, which leads to, e.g., $\Zfs=p^{is}Z_{\tilde{f}}(s)$.
\end{remark}

\begin{definition}[Newton polyhedron]\label{def_NP}
Let $p$ be a prime number. We denote, for $\omega=(\omega_1,\ldots,\omega_n)\in\N^n$, by $x^{\omega}$ the corresponding monomial $x_1^{\omega_1}\cdots x_n^{\omega_n}$ in $\Zp[x_1,\ldots,x_n]$. Let $f(x)=f(x_1,\ldots,x_n)=\sum_{\omega\in\N^n}a_{\omega}x^{\omega}$ be a nonzero polynomial over \Zp, satisfying $f(0)=0$. Let $\Rplus=\{x\in\R\mid x\geqslant0\}$ and $\supp(f)=\{\omega\in\N^n\mid a_{\omega}\neq0\}$, the support of $f$. The Newton polyhedron \Gf\ of $f$ is defined as the convex hull in \Rplusn\ of the set
\begin{equation*}
\bigcup_{\omega\in\supp(f)}\omega+\Rplusn.
\end{equation*}
The global Newton polyhedron $\Gglf$ of $f$ is defined as the convex hull of $\supp(f)$. Clearly, we have $\Gf=\Gglf+\Rplusn$.

Let $\ff=(f_1,\ldots,f_t):\Qpn\to\Qp^t$ be a nonconstant polynomial mapping, with $f_1,\ldots,f_t\in\Zp[x_1,\ldots,x_n]$, satisfying $\ff(0)=0$. We define the support $\supp(\ff)$ of \ff\ as the union of the supports of its composing polynomials, and its Newton polyhedron \Gff, as above, as the convex hull in \Rplusn\ of $\bigcup_{\omega\in\supp(\ff)}\omega+\Rplusn$.

Let $\I\lhd\Zp[x_1,\ldots,x_n]$ be a nonzero proper monomial ideal\footnote{A monomial ideal is an ideal that can be generated by monic monomials.}. We define the Newton polyhedron \GI\ of \I\ as the convex hull in \Rplusn\ of those $\omega\in\N^n$, such that $x^{\omega}\in \I$. The polyhedron \GI\ coincides with the Newton polyhedron of the monomial mapping $(x^{\omega_1},\ldots,x^{\omega_t})$ for any set $\{x^{\omega_1},\ldots,x^{\omega_t}\}$ of monomial generators of \I.
\end{definition}

\begin{notation}
For $a\in\Zpn$, we denote by $\overline{a}=a+(p\Zp)^n\in\Fpn$ its reduction modulo $(p\Zp)^n$. For $f\in\Zp[x_1,\ldots,x_n]$, we denote by $\overline{f}$ the polynomial over \Fp, obtained from $f$, by reducing each of its coefficients modulo $p\Zp$. Analogously, for a polynomial mapping $\ff=(f_1,\ldots,f_t)$, we denote $\overline{\ff}=(\overline{f_1},\ldots,\overline{f_t})$.
\end{notation}

\begin{definition}[Non-degenerated]\label{def_non-degenerated}
Let $f$ be as in Definition \ref{def_NP}. For every face\footnote{By a face of $\Gf$ we mean $\Gf$ itself or one of its proper faces, which are the intersections of $\Gf$ with a supporting hyperplane. See, e.g., \cite{Roc70}.} $\tau$ of the Newton polyhedron $\Gf$ of $f$, we put
\begin{equation*}
\ft(x)=\sum_{\omega\in\tau}a_{\omega}x^{\omega}.
\end{equation*}

We say that $f$ is non-degenerated over \Fp\ with respect to (all the faces of) its Newton polyhedron \Gf, if for every\footnote{Thus also for \Gf.} face $\tau$ of \Gf, the zero locus of the polynomial \fbart\ has no singularities in \Fpcrossn, or, equivalently, the set of congruences
\begin{equation*}
\left\{
\begin{aligned}
\ft(x)&\equiv0\bmod p,\\
\frac{\partial \ft}{\partial x_i}(x)&\equiv0\bmod p;\quad i=1,\ldots,n;
\end{aligned}
\right.
\end{equation*}
has no solutions in \Zpxn.
\end{definition}

%
\begin{definition}[Strongly non-degenerated]\label{def_strong_non-degenerated}
Let \ff\ be as in Definition \ref{def_NP}. For every face $\tau$ of the Newton polyhedron $\Gff$ of \ff, we denote $\fft=(f_{1,\tau},\ldots,f_{t,\tau})$, with the $f_{i,\tau}$ defined in the same way as the \ft\ in Definition~\ref{def_non-degenerated}.

We call the mapping \ff\ strongly non-degenerated over the field \Fp\ with respect to (all the faces of) its Newton polyhedron \Gff, if for every face $\tau$ of \Gff\ and all $a\in\Zpxn$, satisfying $\fft(a)\equiv0\bmod p$, the Jacobian matrix $J(\fft,a)=\left((\partial f_{i,\tau}/\partial x_j)(a)\right)_{i,j}$ has maximal rank (equal to $\min(t,n)$) modulo $p$.\footnote{Meaning that the matrix $\left(\overline{(\partial f_{i,\tau}/\partial x_j)(a)}\right)_{i,j}$ over \Fp\ has maximal rank.} In other words: \ff\ is strongly non-degenerated over \Fp\ with respect to \Gff, if for every face $\tau$ of \Gff\ and all $\overline{a}\in\Fpcrossn$ in the zero locus of $\overline{\fft}$, the Jacobian matrix $J(\overline{\fft},\overline{a})=\left((\partial \overline{f_{i,\tau}}/\partial x_j)(\overline{a})\right)_{i,j}$ over \Fp, has maximal rank.
\end{definition}

\begin{definition}[$m_f(k)$]\label{def_mf}
Let $f$ be as in Definition \ref{def_NP}. For $k\in\Rplusn$, we define
\begin{equation*}
m_f(k)=\inf_{x\in\Gf}k\cdot x,
\end{equation*}
where $k\cdot x$ denotes the scalar product of $k$ and $x$.
\end{definition}

The infimum in the definition above is actually a minimum, where the minimum can as well be taken over the global Newton polyhedron \Gglf\ of $f$, which is a compact set, or even over the finite set $\supp(f)$.

\begin{definition}[First meet locus]\label{def_firstmeetlocus}
Let $f$ be as in Definition \ref{def_NP} and $k\in\Rplusn$. We define the first meet locus of $k$ as the face
\begin{equation*}
F_f(k)=\{x\in\Gf\mid k\cdot x=m_f(k)\}
\end{equation*}
of \Gf.
\end{definition}

\begin{definition}[Primitive vector]
A vector $k\in\Rn$ is called primitive if the components of $k$ are integers whose greatest common divisor is one.
\end{definition}

\begin{definition}[\Df]\label{def_Df}
Let $f$ be as in Definition \ref{def_NP}. For a face $\tau$ of \Gf, we call
\begin{equation*}
\Dft=\{k\in\Rplusn\mid F_f(k)=\tau\}
\end{equation*}
the cone associated to $\tau$. The \Dft\ are the equivalence classes of the equivalence relation $\sim_f$ on \Rplusn, defined by
\begin{equation*}
k\sim_f k'\qquad\textrm{if and only if}\qquad F_f(k)=F_f(k').
\end{equation*}
The \lq cones\rq\ \Dft\ thus form a partition of \Rplusn\ which we denote by \Df; i.e.,
\begin{equation*}
\Df=\{\Dft\mid \tau\ \mathrm{is\ a\ face\ of}\ \Gf\}=\Rplusn/\sim_f.
\end{equation*}
\end{definition}

The \Dft\ are in fact relatively open\footnote{A subset of \Rplusn\ is called relatively open if it is open in its affine closure.} convex cones\footnote{A subset $C$ of \Rn\ is called a convex cone if it is a convex set and $\lambda x\in C$ for all $x\in C$ and all $\lambda\in\Rplusnul$.} with a very specific structure, as stated in the following lemma.

\begin{lemma}\label{lemma_struc_Dft} \cite[Lemma 2.6]{DH01}. Let $f$ be as in Definition~\ref{def_NP}. Let $\tau$ be a proper face of \Gf\ and let $\gamma_1,\ldots,\gamma_r$ be the facets\footnote{A facet is a face of codimension one.} of \Gf\ that contain $\tau$. Let $k_1,\ldots,k_r$ be the unique primitive vectors in $\N^n\setminus\{0\}$ that are perpendicular to $\gamma_1,\ldots,\gamma_r$, respectively. Then the cone \Dft\ associated to $\tau$ is the convex cone
\begin{equation*}
\Dft=\{\lambda_1k_1+\lambda_2k_2+\cdots+\lambda_rk_r\mid \lambda_i\in\Rplusnul\},
\end{equation*}
and its dimension\footnote{The dimension of a convex cone is the dimension of its affine hull.} equals $n-\dim\tau$.
\end{lemma}

\begin{definition}\label{def_rationalcone}
For $k_1,\ldots,k_r\in\Rn\setminus\{0\}$, we call $\Delta=\{\lambda_1k_1+\lambda_2k_2+\cdots+\lambda_rk_r\mid \lambda_i\in\Rplusnul\}$ the cone strictly positively spanned by the vectors $k_1,\ldots,k_r$. When the $k_1,\ldots,k_r$ can be chosen from \Zn, we call it a rational cone. If we can choose $k_1,\ldots,k_r$ linearly independent over \R, $\Delta$ is called a simplicial cone. If $\Delta$ is rational and $k_1,\ldots,k_r$ can be chosen from a \Z-module basis of \Zn, we call $\Delta$ a simple cone.
\end{definition}

It follows from Lemma~\ref{lemma_struc_Dft} that the topological closures $\overbar{\Dft}$\footnote{$\overbar{\Dft}=\{\lambda_1k_1+\lambda_2k_2+\cdots+\lambda_rk_r\mid \lambda_i\in\Rplus\}=\{k\in\Rplusn\mid F_f(k)\supset\tau\}$.} of the cones \Dft\ form a fan\footnote{A fan $\mathcal{F}$ is a finite set of rational polyhedral cones such that every face of a cone in $\mathcal{F}$ is contained in $\mathcal{F}$ and the intersection of each two cones $C$ and $C'$ in $\mathcal{F}$ is a face of both $C$ and $C'$.} $\overbar{\Df}$ of rational polyhedral cones\footnote{A rational polyhedral cone is a closed convex cone, generated by a finite subset of \Zn.}.

\begin{remark}
The function $m_f$ from Definition~\ref{def_mf} is linear on each cone $\overbar{\Dft}$.
\end{remark}

\begin{remark}
Let \ff\ and \I\ be as in Definition~\ref{def_NP}. Let $k\in\Rplusn$ and let $\tau$ and $\tau'$ be faces of \Gff\ and \GI, respectively. We then have analogous definitions and results for $m_{\ff}(k)$, $F_{\ff}(k)$, $\sim_{\ff}$, $\Dff(\tau)$, \Dff, and $\overbar{\Dff}$, associated to \ff, and for $m_{\I}(k)$, $F_{\I}(k)$, $\sim_{\I}$, $\DI(\tau')$, \DI, and $\overbar{\DI}$, associated to \I.
\end{remark}

We state without proofs, the following two lemmas (see, e.g., \cite{DH01}).

\begin{lemma}
Let $\Delta$ be the cone strictly positively spanned by the vectors $k_1,\ldots,k_r\allowbreak\in\Rplusn\setminus\{0\}$. Then there exists a finite partition of $\Delta$ into cones $\Delta_i$, such that each $\Delta_i$ is strictly positively spanned by a \R-linearly independent subset of $\{k_1,\ldots,k_r\}$. We call such a decomposition a simplicial decomposition of $\Delta$ without introducing new rays.
\end{lemma}

\begin{lemma}
Let $\Delta$ be a rational simplicial cone. Then there exists a finite partition of $\Delta$ into simple cones. (In general, such a decomposition requires the introduction of new rays.)
\end{lemma}

Finally, we need the following notion, which is closely related to the notion of a simple cone.

\begin{definition}\label{def_mult}
Let $k_1,\ldots,k_r$ be \Q-linearly independent vectors in \Zn. The multiplicity of $k_1,\ldots,k_r$, denoted by $\mult(k_1,\ldots,k_r)$, is defined as the index of the lattice $\Z k_1+\cdots+\Z k_r$ in the group of points with integral coordinates in the subspace spanned by $k_1,\ldots,k_r$ of the \Q-vector space $\Q^n$. One can check that this number equals the cardinality of the set
\begin{equation*}
\Zn\cap\left\{\sum\nolimits_{i=1}^r\lambda_ik_i\;\middle\vert\;0\leqslant\lambda_i<1\text{ for }i=1,\ldots,r\right\}.
\end{equation*}

Let $\Delta$ be the cone strictly positively spanned by $k_1,\ldots,k_r$. We define the multiplicity of $\Delta$ as the multiplicity of $k_1,\ldots,k_r$, and we denote it by $\mult\Delta$.
\end{definition}

\begin{remark}
Let $\Delta$ be as in Definition~\ref{def_mult}. Note that $\Delta$ is simple if and only if $\mult\Delta=\mult(k_1,\ldots,k_r)=1$.
\end{remark}

\section{Igusa's \except{toc}{local }zeta function of a monomial ideal and a polynomial measure}\label{formule1}
In the previous section we met partitions \Df, \Dff, and \DI\ of \Rplusn, associated to a polynomial, a polynomial mapping, and a monomial ideal, respectively. When, as in this section, we are dealing with two polynomials $f$ and $g$, a polynomial mapping \ff\ and a polynomial $g$, or an ideal \I\ and a polynomial $g$, we will also consider a partition of \Rplusn\ by finitely and strictly positively spanned rational cones, associated to both $f$ and $g$, both \ff\ and $g$, or both \I\ and $g$.

\begin{definition}[\Dfg, \Dffg, \DIg]\label{def_Dfg}
Let $f$, \ff, \I, and $g$ be as in Definition~\ref{def_NP}. The partition \Dfg\ of \Rplusn, associated to $f$ and $g$, will consist of all the nonempty intersections of cones in \Df\ with cones in \Dg; i.e.,
\begin{multline*}
\Dfg=\{\Dft\cap\Dg(\tau')\mid\\\tau\textrm{ is a face of }\Gf,\ \tau'\textrm{ is a face of }\Gg,\textrm{ and }\Dft\cap\Dg(\tau')\neq\emptyset\}.
\end{multline*}
The set of cones \Dfg\ will then be the quotient of \Rplusn\ by the equivalence relation
\begin{equation*}
k\sim_{f,g} k'\qquad\textrm{if and only if}\qquad F_f(k)=F_f(k')\textrm{ and }F_g(k)=F_g(k').
\end{equation*}
The partitions \Dffg, associated to \ff\ and $g$, and \DIg, associated to \I\ and $g$, are defined in the same way.
\end{definition}

%
%
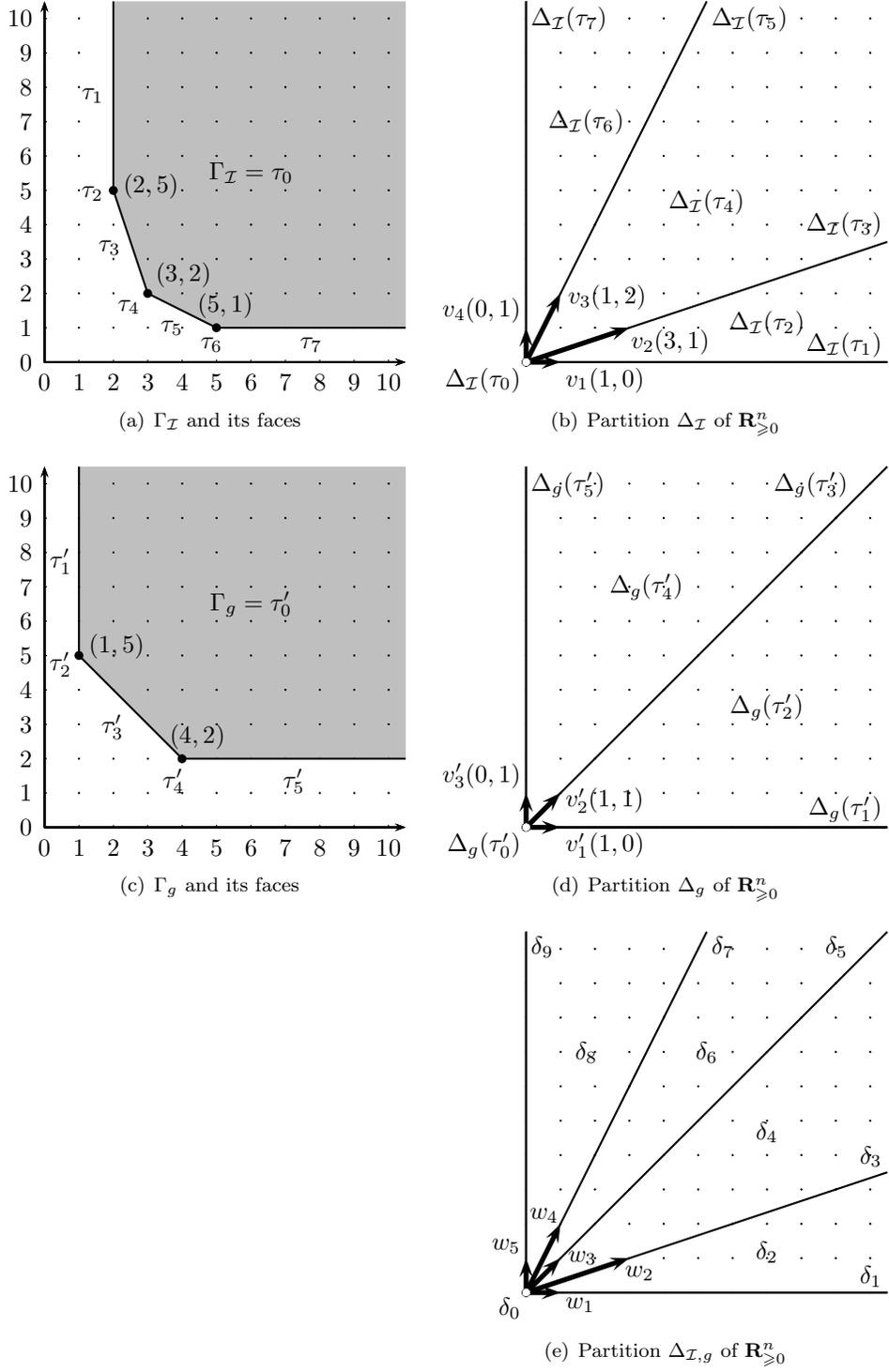
\begin{figure}
\psset{unit=.0385\textwidth}
\centering
\subfigure[\GI\ and its faces]{
\begin{pspicture}(-1,-1)(10.6,10.6)
\pspolygon*[linecolor=lightgray,linewidth=0pt](2,10.5)(2,5)(3,2)(5,1)(10.5,1)(10.5,10.5)
\psgrid[subgriddiv=1,griddots=1,gridlabels=0pt](10,10)
\psaxes[ticksize=1pt]{->}(0,0)(0,0)(10.5,10.5)
\psdots(2,5)(3,2)(5,1)
\psline(2,10.5)(2,5)(3,2)(5,1)(10.5,1)
\uput{4pt}[l](2,7.75){$\tau_1$}
\uput{4pt}[190](2,5){$\tau_2$}
\uput{4pt}[199](2.5,3.5){$\tau_3$}
\uput{4pt}[222](3,2){$\tau_4$}
\uput{4pt}[244](4,1.5){$\tau_5$}
\uput{4pt}[257](5,1){$\tau_6$}
\uput{4pt}[d](7.75,1){$\tau_7$}
\rput(6,5.5){$\GI=\tau_0$}
\uput{4pt}[10](2,5){$(2,5)$}
\uput{4pt}[42](3,2){$(3,2)$}
\uput{4pt}[77](5,1){$(5,1)$}
\end{pspicture}
}\hfill
\subfigure[Partition \DI\ of \Rplusn]{
\begin{pspicture}(-2.4,-1)(10.5,10.6)
\psgrid[subgriddiv=1,griddots=1,gridlabels=0pt](10,10)
\psaxes[labels=none,ticks=none](0,0)(0,0)(10.5,10.5)
\psline(5.25,10.5)(0,0)(10.5,3.5)
\psline[linewidth=2pt]{->}(0,0)(0,1)
\psline[linewidth=2pt]{->}(0,0)(1,2)
\psline[linewidth=2pt]{->}(0,0)(3,1)
\psline[linewidth=2pt]{->}(0,0)(1,0)
\psdot[dotstyle=o](0,0)
\uput{3pt}[ul](10.5,0){$\DI(\tau_1)$}
\rput(7,1.12){$\DI(\tau_2)$}
\uput{3pt}[ul](10.5,3.5){$\DI(\tau_3)$}
\rput(5.25,4.67){$\DI(\tau_4)$}
\uput{3pt}[dr](5.25,10.5){$\DI(\tau_5)$}
\rput(1.75,7){$\DI(\tau_6)$}
\uput{3pt}[dr](0,10.5){$\DI(\tau_7)$}
\uput{3pt}[dl](0,0){$\DI(\tau_0)$}
\uput{3pt}[ul](0,1){$v_4(0,1)$}
\uput{3pt}[-10](1,2){$v_3(1,2)$}
\uput{1pt}[dr](3,1){$v_2(3,1)$}
\uput{3pt}[dr](1,0){$v_1(1,0)$}
\end{pspicture}
}\\
\subfigure[\Gg\ and its faces]{
\begin{pspicture}(-1,-1)(10.6,10.6)
\pspolygon*[linecolor=lightgray,linewidth=0pt](1,10.5)(1,5)(4,2)(10.5,2)(10.5,10.5)
\psgrid[subgriddiv=1,griddots=1,gridlabels=0pt](10,10)
\psaxes[ticksize=1pt]{->}(0,0)(0,0)(10.5,10.5)
\psdots(1,5)(4,2)
\psline(1,10.5)(1,5)(4,2)(10.5,2)
\uput{3pt}[l](1,7.75){$\tau_1'$}
\uput{3pt}[203](1,5){$\tau_2'$}
\uput{4pt}[225](2.5,3.5){$\tau_3'$}
\uput{4pt}[248](4,2){$\tau_4'$}
\uput{4pt}[d](7.25,2){$\tau_5'$}
\rput(6,6.5){$\Gg=\tau_0'$}
\uput{4pt}[23](1,5){$(1,5)$}
\uput{4pt}[68](4,2){$(4,2)$}
\end{pspicture}
}\hfill
\subfigure[Partition \Dg\ of \Rplusn]{
\begin{pspicture}(-2.4,-1)(10.5,10.6)
\psgrid[subgriddiv=1,griddots=1,gridlabels=0pt](10,10)
\psaxes[labels=none,ticks=none](0,0)(0,0)(10.5,10.5)
\psline(0,0)(10.5,10.5)
\psline[linewidth=2pt]{->}(0,0)(0,1)
\psline[linewidth=2pt]{->}(0,0)(1,1)
\psline[linewidth=2pt]{->}(0,0)(1,0)
\psdot[dotstyle=o](0,0)
\uput{3pt}[ul](10.5,0){$\Dg(\tau_1')$}
\rput(7,3.5){$\Dg(\tau_2')$}
\uput{3pt}[dl](9.5,10.5){$\Dg(\tau_3')$}
\rput(3.5,7){$\Dg(\tau_4')$}
\uput{3pt}[dr](0,10.5){$\Dg(\tau_5')$}
\uput{3pt}[dl](0,0){$\Dg(\tau_0')$}
\uput{3pt}[ul](0,1){$v_3'(0,1)$}
\uput{2pt}[-23](1,1){$v_2'(1,1)$}
\uput{3pt}[dr](1,0){$v_1'(1,0)$}
\end{pspicture}
}\\
\hfill\subfigure[Partition \DIg\ of \Rplusn]{
\begin{pspicture}(-2.4,-1)(10.5,10.6)
\psgrid[subgriddiv=1,griddots=1,gridlabels=0pt](10,10)
\psaxes[labels=none,ticks=none](0,0)(0,0)(10.5,10.5)
\psline(5.25,10.5)(0,0)(10.5,3.5)
\psline(0,0)(10.5,10.5)
\psline[linewidth=2pt]{->}(0,0)(0,1)
\psline[linewidth=2pt]{->}(0,0)(1,1)
\psline[linewidth=2pt]{->}(0,0)(1,2)
\psline[linewidth=2pt]{->}(0,0)(3,1)
\psline[linewidth=2pt]{->}(0,0)(1,0)
\psdot[dotstyle=o](0,0)
\uput{3pt}[ul](10.5,0){$\delta_1$}
\rput(7,1.167){$\delta_2$}
\uput{3pt}[ul](10.5,3.5){$\delta_3$}
\rput(7,4.667){$\delta_4$}
\uput{3pt}[dl](9.5,10.5){$\delta_5$}
\rput(5.25,7){$\delta_6$}
\uput{3pt}[dr](5.25,10.5){$\delta_7$}
\rput(1.75,7){$\delta_8$}
\uput{3pt}[dr](0,10.5){$\delta_9$}
\uput{3pt}[dl](0,0){$\delta_0$}
\uput{3pt}[ul](0,1){$w_5$}
\uput{1pt}[ul](1,2){$w_4$}
\uput{3pt}[2](1,1){$w_3$}
\uput{3pt}[-65](3,1){$w_2$}
\uput{3pt}[dr](1,0){$w_1$}
\end{pspicture}
}
\caption{Newton polyhedra and partitions of \Rplusn\ associated to $\I=(x^5y,x^3y^2,x^2y^5)$ and $g=x^4y^2+xy^5$}\label{fig_vb}
\end{figure}
%
%

\begin{example}\label{voorbeeld}
Consider the monomial ideal $\I=(x^5y,x^3y^2,x^2y^5)\lhd\Zp[x,y]$ and the polynomial $g(x,y)=x^4y^2+xy^5\in\Zp[x,y]$. The Newton polyhedra \GI\ and \Gg\ and the partitions \DI, \Dg, and \DIg\ of \Rplusn\ are drawn in Figure~\ref{fig_vb}. Note that the vectors $v_1,\ldots,v_4;v_1',\ldots,v_3';w_1,\ldots,w_5$ are perpendicular to the faces associated to the rays\footnote{A ray is a cone of dimension one.} they span. One checks that $g$ is non-degenerated over \Fp\ with respect to its Newton polyhedron, if and only if $p\neq3$.
\end{example}

Let \I\ be a nonzero proper monomial ideal of $\Zp[x_1,\ldots,x_n]$ and $g$ a non-zero polynomial in $\Zp[x_1,\ldots,x_n]$ with $g(0)=0$, which is non-degenerated over \Fp\ with respect to all the faces of its Newton polyhedron. We consider the Newton polyhedra \GI\ and \Gg\ of \I\ and $g$, and the corresponding partitions \DI, \Dg, and \DIg\ of \Rplusn.

\begin{notation}
From the definition of \DIg\ and the fact that \DI\ and \Dg\ consist of disjoint sets, it follows that each cone $\delta$ in \DIg\ can be written in a unique way as the intersection $\delta=\DIt\cap\Dg(\tau')$ of a cone \DIt\ in \DI\ with a cone $\Dg(\tau')$ in \Dg. So to each cone $\delta\in\DIg$ we can associate a face $\tau$ of \GI\ and a face $\tau'$ of \Gg. For $\delta\in\DIg$, denote by \gd\ the polynomial\footnote{See Definition~\ref{def_non-degenerated}.} $g_{\tau'}$, with $\tau'$ the face of \Gg\ associated to $\delta$. Finally, for $\delta\in\DIg$, put
\begin{equation*}
\Nd=\#\left\{a\in\Fpcrossn\mid\gbard(a)=0\right\}.
\end{equation*}
\end{notation}

\begin{notation}
For $k=(k_1,\ldots,k_n)\in\Rn$, we denote
\begin{equation*}
\sigma(k)=\sum_{i=1}^nk_i.
\end{equation*}
\end{notation}

\begin{theorem}\label{theformule1}
Let \I\ and $g$ be as above. Then we have:
\begin{equation*}
\ZIgs=\sum_{\delta\in\DIg}\Ld\Sd,
\end{equation*}
with
\begin{equation*}
\Ld=p^{-n}\left((p-1)^n-\Nd\frac{p}{p+1}\right)
\end{equation*}
and
\begin{equation*}
\Sd=\sum_{k\in\Nn\cap\delta}p^{-m_{\I}(k)s-m_g(k)-\sigma(k)},
\end{equation*}
for every cone $\delta$ in \DIg.

The above formula for \Sd\ can be rewritten as a rational expression in $p^{-s}$. Consider therefore a partition of the rational cone $\delta$ into rational simplicial cones $\delta_i$, $i\in I$, without the introduction of new rays. For each $i$, let $\delta_i$ be the cone strictly positively spanned by the linearly independent, primitive vectors $k_{i,1},\ldots,k_{i,r_i}\in\Nn\setminus\{0\}$. Then
\begin{equation*}
\Sd=\sum_{i\in I}\frac{\sum_hp^{m_{\I}(h)s+m_g(h)+\sigma(h)}}{\prod_{j=1}^{r_i}(p^{m_{\I}(k_{i,j})s+m_g(k_{i,j})+\sigma(k_{i,j})}-1)},
\end{equation*}
where $h$ runs through the elements of the set
\begin{equation*}
\Zn\cap\left\{\sum\nolimits_{j=1}^{r_i}\lambda_jk_{i,j}\;\middle\vert\;0\leqslant\lambda_j<1\text{ for }j=1,\ldots,r_i\right\}.
\end{equation*}

From the formula for \Ld\ and the rational expression for \Sd, it then follows that the real parts of the candidate poles of \ZIg\ are given by the rational numbers
\begin{equation*}
-\frac{m_g(k)+\sigma(k)}{m_{\I}(k)},
\end{equation*}
for $k$ a primitive generator of a ray in \DIg.
\end{theorem}

\begin{proof}
The proof is similar to the proofs of \cite[Theorem~4.2]{DH01} and \cite[Proposition~2.1]{HMY07}. First we divide the integration domain \Zpn, based on the partition \DIg\ of \Rplusn\ associated to \I\ and $g$:
\begin{align*}
\ZIgs&=\int_{\Zpn}|\I(x)|^s|g(x)||dx|\\
&=\sum_{\delta\in\DIg}\sum_{k\in\Nn\cap\delta}\int_{\substack{x\in\Zpn\\\ord x=k}}|\I(x)|^s|g(x)||dx|.
\end{align*}
Recall that $|\I(x)|=\max\{|f(x)|\mid f\in \I\}=p^{-\min\{\ord f(x)\mid f\in \I\}}$. It's easy to see that this minimum is always attained in one of the monic monomials that generate \I. If $x^{\omega}$ is a monomial in $\Zp[x_1,\ldots,x_n]$, then for $x\in\Zpn$ with $\ord x=k$, the order of $x^{\omega}$ equals $k\cdot\omega$. So, since $\min\{k\cdot\omega\mid \omega\in\supp(I)\}=m_{\I}(k)$, it follows that
\begin{equation*}
\ZIgs=\sum_{\delta\in\DIg}\sum_{k\in\Nn\cap\delta}p^{-m_{\I}(k)s}\int_{\substack{x\in\Zpn\\\ord x=k}}|g(x)||dx|.
\end{equation*}
From the proof of \cite[Theorem 4.2]{DH01}, it follows that the integral in the previous equation is equal to $p^{-m_g(k)-\sigma(k)}\Ld$. Hence
\begin{equation*}
\ZIgs=\sum_{\delta\in\DIg}\sum_{k\in\Nn\cap\delta}p^{-m_{\I}(k)s}p^{-m_g(k)-\sigma(k)}\Ld,
\end{equation*}
and since \Ld\ is independent of $k$, we find
\begin{align*}
\ZIgs&=\sum_{\delta\in\DIg}\Ld\sum_{k\in\Nn\cap\delta}p^{-m_{\I}(k)s-m_g(k)-\sigma(k)}\\
&=\sum_{\delta\in\DIg}\Ld\Sd.
\end{align*}

For the easy proof of the rational formula for \Sd, we refer to \cite[Theorem~4.2]{DH01}. Essential here is the fact that the functions $m_{\I}$ and $m_g$ are linear on each $\delta$.
\end{proof}

\begin{example}[Continuation of Example~\ref{voorbeeld}] We will now calculate Igusa's local zeta function \ZIg\ for the monomial ideal \I\ and the polynomial $g$ from Example~\ref{voorbeeld}. One can verify that $g$ is non-degenerated over \Fp\ with respect to \Gg, if and only if $p\neq3$. So we restrict to this case and use the formula from Theorem~\ref{theformule1} to calculate \ZIgs.

%
%
\begin{table}
\centering
{
\setlength{\belowrulesep}{.65ex}
\setlength{\aboverulesep}{.4ex}
\setlength{\belowbottomsep}{1.3ex}
\setlength{\defaultaddspace}{.5em}
\begin{tabular}{*{5}{l}}\toprule
primitive generator & \multirow{2}*{$m_{\I}(w)$} & \multirow{2}*{$m_g(w)$} & \multirow{2}*{$\sigma(w)$} & real candidate pole\\
$w$ of a ray in \DIg & & & & associated to $w$\\\midrule
$w_1(1,0)$ & $2$ & $1$ & $1$ & $-1$\\\addlinespace
$w_2(3,1)$ & $11$ & $8$ & $4$ & $-12/11$\\\addlinespace
$w_3(1,1)$ & $5$ & $6$ & $2$ & $-8/5$\\\addlinespace
$w_4(1,2)$ & $7$ & $8$ & $3$ & $-11/7$\\\addlinespace
$w_5(0,1)$ & $1$ & $2$ & $1$ & $-3$\\\midrule[\heavyrulewidth]
integral point $h$ & $m_{\I}(h)$ & $m_g(h)$ & $\sigma(h)$ & \\\cmidrule{1-4}
$(2,1)$ & $8$ & $7$ & $3$ & \\
\cmidrule[\heavyrulewidth]{1-4}\addlinespace[\belowrulesep]
\end{tabular}
}
\caption{Data associated to $(2,1)$ and the primitive generators of the rays in \DIg}\label{tabelrays}
\end{table}
%
%

%
%
\begin{table}
\centering
{
\setlength{\belowrulesep}{.8ex}
\setlength{\aboverulesep}{.7ex}
\setlength{\belowbottomsep}{1.3ex}
\setlength{\defaultaddspace}{.57em}
\begin{tabular}{*{7}{l}}\toprule
cone $\delta$ & $\dim$ & primitive & $\mult$ & \multirow{2}*{\Nd} & \multirow{2}*{\Ld} & \multirow{2}*{\Sd}\\
of \DIg & $\delta$ & generators & $\delta$ & & & \\\midrule
$\delta_0$ & $0$ & -- & $1$ & $N_g$ & $\bigl(\frac{p-1}{p}\bigr)^2-\frac{N_g}{p(p+1)}$ & $1$\\\addlinespace
$\delta_1$ & $1$ & $(1,0)$ & $1$ & $0$ & $\bigl(\frac{p-1}{p}\bigr)^2$ & $\frac{1}{p^{2s+2}-1}$\\\addlinespace
$\delta_2$ & $2$ & $(1,0),(3,1)$ & $1$ & $0$   & $\bigl(\frac{p-1}{p}\bigr)^2$ & $\frac{1}{(p^{2s+2}-1)(p^{11s+12}-1)}$\\\addlinespace
$\delta_3$ & $1$ & $(3,1)$ & $1$ & $0$ & $\bigl(\frac{p-1}{p}\bigr)^2$ & $\frac{1}{p^{11s+12}-1}$\\\addlinespace
$\delta_4$ & $2$ & $(3,1),(1,1)$ & $2$ & $0$ & $\bigl(\frac{p-1}{p}\bigr)^2$ & $\frac{1+p^{8s+10}}{(p^{11s+12}-1)(p^{5s+8}-1)}$\\\addlinespace
$\delta_5$ & $1$ & $(1,1)$ & $1$ & $N_g$ & $\bigl(\frac{p-1}{p}\bigr)^2-\frac{N_g}{p(p+1)}$ & $\frac{1}{p^{5s+8}-1}$\\\addlinespace
$\delta_6$ & $2$ & $(1,1),(1,2)$ & $1$ & $0$ & $\bigl(\frac{p-1}{p}\bigr)^2$ & $\frac{1}{(p^{5s+8}-1)(p^{7s+11}-1)}$\\\addlinespace
$\delta_7$ & $1$ & $(1,2)$ & $1$ & $0$ & $\bigl(\frac{p-1}{p}\bigr)^2$ & $\frac{1}{p^{7s+11}-1}$\\\addlinespace
$\delta_8$ & $2$ & $(1,2),(0,1)$ & $1$ & $0$ & $\bigl(\frac{p-1}{p}\bigr)^2$ & $\frac{1}{(p^{7s+11}-1)(p^{s+3}-1)}$\\\addlinespace
$\delta_9$ & $1$ & $(0,1)$ & $1$ & $0$ & $\bigl(\frac{p-1}{p}\bigr)^2$ & $\frac{1}{p^{s+3}-1}$\\\bottomrule
\end{tabular}
}
\caption{Data associated to the cones of \DIg}\label{tabelkegels}
\end{table}
%
%

Tables~\ref{tabelrays} and \ref{tabelkegels} show the data related to the cones in \DIg\ and the primitive generators of its rays, needed to \lq fill in\rq\ our formula. The number $N_g$, appearing in Table~\ref{tabelkegels}, is the number of elements in the locus $\{(x,y)\in(\Fpcross)^2\mid\overline{g}(x,y)=x^4y^2+xy^5=0\}$. One can check that this number equals
\begin{equation*}
N_g=
\begin{cases}
3(p-1),&\text{if $p\equiv1,7\bmod12$;}\\
p-1,&\text{if $p\equiv2,5,11\bmod12$.}
\end{cases}
\end{equation*}
The calculations are relatively easy, because, since we work in dimension $2$, all cones $\delta\in\DIg$ are simplicial. Moreover, all cones are simple, except $\delta_4$. Consequently, we have to take the single point
\begin{equation*}
(2,1)\in\Zn\cap\{\lambda(3,1)+\mu(1,1)\mid0\leqslant\lambda,\mu<1\}
\end{equation*}
into account, in the calculation of $S_{\delta_4}$.

For $p\equiv1,7\bmod12$, we find
\begin{equation*}
\ZIgs=\sum_{\delta\in\DIg}\Ld\Sd=\frac{p^6(p-1)A(p,t)}{(p+1)(p^2-t^2)(p^{12}-t^{11})(p^8-t^5)(p^{11}-t^7)(p^3-t)},
\end{equation*}
where $t=p^{-s}$ and $A(p,t)$ is the following polynomial in $p$ and $t$:
\begin{multline*}
A(p,t)=(-p^{2}-3p+1)t^{21}+(p^{5}+3p^{4}-p^{3})t^{20}+(-p^{5}+p^{4}+4p^{3}-p^{2})t^{19}\\*
\shoveright{+(-p^{7}-3p^{6}+p^{5}-p^{4}+p^{2})t^{18}+(2p^{7}-2p^{5})t^{17}+(p^{6}-p^{4})t^{16}}\\*
\shoveright{+(-p^{9}+p^{7}-p^{6}+p^{4})t^{15}+(p^{13}+3p^{12}-p^{11}+p^{9}-p^{7})t^{14}-3p^{15}t^{13}}\\*
\shoveright{+(-p^{15}-3p^{14}+p^{13})t^{12}+(3p^{17}+p^{15}-p^{13})t^{11}+(-p^{18}+p^{16}+p^{14}+3p^{13}-p^{12})t^{10}}\\*
\shoveright{+(-2p^{17}-3p^{16}+2p^{15})t^{9}+(p^{20}-p^{18}+2p^{17}-5p^{15})t^{8}+(-p^{20}+4p^{18})t^{7}}\\*
+(-p^{19}+p^{17})t^{6}+(-p^{25}-3p^{24}+p^{23})t^{3}+3p^{27}t^{2}+3p^{26}t+p^{30}-3p^{29}-p^{28}.
\end{multline*}
Note that none of the candidate poles cancel in this case. For $p\equiv2,5,11\bmod12$, we find a similar result.
\end{example}

\section{Computation of an important integral using Hensel's Lemma}\label{sectHensel}
In order to prove our second main theorem in the next section, we will need a formula for the integral
\begin{equation*}
\int_{\Zpxn}|f(x)|^s|g(x)||dx|,
\end{equation*}
where $f$ and $g$ are polynomials over \Zp\ in $n$ variables. In Corollary~\ref{corol3} we will find a reasonable formula under a number of non-degeneracy conditions on $f$ and $g$ that will be explained further on in this section. The result demands some preparation involving Hensel's Lemma.

\begin{lemma}[Hensel's lemma in several variables]\label{lemma_hensel}
Let $f=(f_1,\ldots,f_n)$ be an $n$-tuple of polynomials $f_i$ in $\Zp[x_1,\ldots,x_n]$. Denote by $J(f)=|\partial f_i/\partial x_j|_{i,j}$ the Jacobian determinant of $f$. Let $k\in\Nnul$, let $a=(a_1,\ldots,a_n)\in\Zpn$ be a $k$-approximate root of $f$ (i.e., $f_i(a)\equiv 0\bmod p^k$ for all $i$), and suppose that $J(f,a)\not\equiv 0\bmod p$. Then there exists a unique root $a'\in\Zpn$ of $f$ near $a$ (i.e., $f_i(a')=0$ and $a'_i\equiv a_i\bmod p^k$ for all $i$).
\end{lemma}

\begin{proof}
For the case $k=1$, we refer to \cite[Section 4.6, Theorem 2]{Bourbaki}. The result for $k>1$ can be proved by applying the statement for $k=1$ to the polynomial $g(y)=p^{-(k-1)}f(a+p^{k-1}y)$.
\end{proof}

\begin{corollary}\label{corol}
Cfr.\ \cite[Corollary 3.4]{DH01}. Let $f(x)\in\Zp[x]$ and $k,l\in\Nnul$ with $k\geqslant l$. Let $a\in\Zp$ such that $f(a)\equiv0\bmod p^l$ and $f'(a)\not\equiv0\bmod p$. Then there exists an element $a'\in\Zp$ such that
\begin{equation*}
\{x\in a+p^l\Zp\mid f(x)\equiv0\bmod p^k\}=a'+p^k\Zp.
\end{equation*}
\end{corollary}

\begin{proof}
By the previous lemma, it follows that there exists a unique root $a'\in\Zp$ of $f$ which is congruent to $a$ modulo $p^l$. One can verify that $a'$ satisfies the required condition.
\end{proof}

\begin{corollary}\label{corol2}
Let $f,g\in\Zp[x_1,x_2]$ and $k\in\Nnul$. Let $a=(a_1,a_2)\in\Zp^2$ such that $f(a)\equiv g(a)\equiv0\bmod p$ and $|J((f,g),a)|\not\equiv0\bmod p$. Then there exists an element $a'=(a_1',a_2')\in\Zp^2$ such that
\begin{equation*}
A_{k,k}=\{x=(x_1,x_2)\in a+(p\Zp)^2\mid f(x)\equiv g(x)\equiv0\bmod p^k\}=a'+(p^k\Zp)^2.
\end{equation*}
\end{corollary}

\begin{proof}
This follows in the same way as Corollary~\ref{corol} from Lemma~\ref{lemma_hensel}.
\end{proof}

It follows that the set $A_{k,k}\subset\Zp^2$ has measure $p^{-2k}$. We will need the following generalization of this result.

\begin{lemma}\label{notAkl}
Let $n\geqslant2$ and $f,g\in\Zp[x_1,\ldots,x_n]$. Let $k,l\in\Nnul$ with $k\geqslant l$. Let $a=(a_1,\ldots,a_n)\in\Zp^n$ such that $f(a)\equiv g(a)\equiv0\bmod p$ and such that the Jacobian matrix $J((f,g),a)$ has rank $2$ modulo $p$. Then the measure of the set
\begin{equation*}
A_{k,l}=\{x=(x_1,\ldots,x_n)\in a+(p\Zp)^n\mid f(x)\equiv0\bmod p^k\ \textrm{and}\ g(x)\equiv0\bmod p^l\}
\end{equation*}
equals $p^{-n-k-l+2}$.
\end{lemma}

\begin{proof}
Since $J((f,g),a)$ has rank $2$ modulo $p$, we can assume, without loss of generality, that
\begin{equation*}
\begin{vmatrix}
{\ds\frac{\partial f}{\partial x_1}(a)}&{\ds\frac{\partial f}{\partial x_2}(a)}\\[+2ex]
{\ds\frac{\partial g}{\partial x_1}(a)}&{\ds\frac{\partial g}{\partial x_2}(a)}
\end{vmatrix}
\not\equiv0\bmod p.
\end{equation*}
This also implies that $(\partial f/\partial x_1)(a)$ and $(\partial f/\partial x_2)(a)$ are not both congruent to zero modulo $p$. Let us again assume that $(\partial f/\partial x_1)(a)\not\equiv0\bmod p$.

Fix $a_3',\ldots,a_n'\in\Zp$ with $a_i'\equiv a_i\bmod p$ for $i=3,\ldots,n$. Applying Corollary~\ref{corol2} to the polynomials $f(x_1,x_2,a_3',\ldots,a_n')$ and $g(x_1,x_2,a_3',\ldots,a_n')$ in the variables $x_1,x_2$ and to $(a_1,a_2)\in\Zp^2$, we find that there exist $a_1',a_2'\in\Zp$, depending on $a_3',\ldots,a_n'$, such that for all $x_3,\ldots,x_n\in\Zp$, satisfying $x_i\equiv a_i'\bmod p^k$ for $i=3,\ldots,n$, one has
\begin{multline*}
\{(x_1,x_2)\in(a_1,a_2)+(p\Zp)^2\mid f(x_1,\ldots,x_n)\equiv g(x_1,\ldots,x_n)\equiv0\bmod p^l\}\\=(a_1',a_2')+(p^l\Zp)^2.
\end{multline*}

Consider now $a_2''\in\Zp$ with $a_2''\equiv a_2'\bmod p^l$. Applying Corollary~\ref{corol} to the polynomial $f(x_1,a_2'',a_3',\ldots,a_n')\in\Zp[x_1]$ and $a_1'\in\Zp$, we find an element $a_1''\in\Zp$, depending on $a_2'',a_3',\ldots,a_n'$, such that for all $x_2,\ldots,x_n\in\Zp$, satisfying $x_2\equiv a_2'',x_i\equiv a_i'\bmod p^k$ for $i=3,\ldots,n$, we have
\begin{equation*}
\{x_1\in a_1'+p^l\Zp\mid f(x_1,\ldots,x_n)\equiv0\bmod p^k\}=a_1''+p^k\Zp.
\end{equation*}

These two considerations show that the set
\begin{equation*}
A_{k,l}=\{x=(x_1,\ldots,x_n)\in a+(p\Zp)^n\mid f(x)\equiv0\bmod p^k\ \textrm{and}\ g(x)\equiv0\bmod p^l\}
\end{equation*}
equals the union
\begin{equation}\label{union}
\bigcup_{\substack{a_3'+p^k\Zp,\ldots,a_n'+p^k\Zp\\a_i'\equiv a_i\bmod p}}\ \ \bigcup_{\substack{a_2''+p^k\Zp\\a_2''\equiv a_2'\bmod p^l}}(a_1'',a_2'',a_3',\ldots,a_n')+(p^k\Zp)^n.
\end{equation}
Hereby the first union is taken over all elements $(a_3'+p^k\Zp,\ldots,a_n'+p^k\Zp)\in(\Zp/p^k\Zp)^{n-2}$ with $a_i'\equiv a_i\bmod p$ for $i=3,\ldots,n$; the second union is over all the cosets $a_2''+p^k\Zp\in\Zp/p^k\Zp$ with $a_2''\equiv a_2'\bmod p^l$, where $a_2'$ depends on $a_3',\ldots,a_n'$ and $a_1''$ depends on $a_2'',a_3',\ldots,a_n'$ as described above.

Because the union~\eqref{union} is a disjoint union, we now find easily that the measure of $A_{k,l}$ equals $p^{(k-1)(n-2)}p^{k-l}p^{-kn}=p^{-n-k-l+2}.$
\end{proof}

%
%
%
Now we can prove the following proposition, which is the key result for the formula we are looking for.

\begin{proposition}\label{prop}
Let $n\geqslant2$ and $f,g\in\Zp[x_1,\ldots,x_n]$. Let $a\in\Zp^n$ be no solution to any of the following sets of congruence relations:
\begin{equation}\label{stelsels}
\begin{gathered}
\left\{
\begin{aligned}
f(x)&\equiv0\bmod p,\\
\frac{\partial f}{\partial x_i}(x)&\equiv0\bmod p;\quad i=1,\ldots,n;
\end{aligned}
\right.
\ \ \quad
\left\{
\begin{aligned}
g(x)&\equiv0\bmod p,\\
\frac{\partial g}{\partial x_i}(x)&\equiv0\bmod p;\quad i=1,\ldots,n;
\end{aligned}
\right.\\
\text{and}\qquad\left\{
\begin{aligned}
f(x)&\equiv0\bmod p,\\
g(x)&\equiv0\bmod p,\\
\begin{vmatrix}
\frac{\partial f}{\partial x_i}(a)&\frac{\partial f}{\partial x_j}(a)\\
\frac{\partial g}{\partial x_i}(a)&\frac{\partial g}{\partial x_j}(a)
\end{vmatrix}
&\equiv0\bmod p;\quad i,j=1,\ldots,n;\ i<j.
\end{aligned}
\right.
\end{gathered}
\end{equation}
Then for $s\in\C$ with $\Re(s)>0$, we have
\begin{equation*}
\int_{a+(p\Zp)^n}|f(x)|^s|g(x)||dx|=
\begin{cases}
p^{-n},&\textrm{\normalfont if $f(a)\not\equiv0$ and $g(a)\not\equiv0\bmod p$;}\\
p^{-n}\frac{p-1}{p^{s+1}-1},&\textrm{\normalfont if $f(a)\equiv0$ and $g(a)\not\equiv0\bmod p$;}\\
p^{-n}\frac{1}{p+1},&\textrm{\normalfont if $f(a)\not\equiv0$ and $g(a)\equiv0\bmod p$;}\\
p^{-n}\frac{p-1}{(p^{s+1}-1)(p+1)},&\textrm{\normalfont if $f(a)\equiv g(a)\equiv0\bmod p$.}
\end{cases}
\end{equation*}
\end{proposition}

\begin{proof}
The first three cases follow easily from \cite[Proposition~3.1]{DH01}. Let us now assume $f(a)\equiv g(a)\equiv0\bmod p$. For $x\in a+(p\Zp)^n$, we have $f(x)\equiv f(a)\equiv g(x)\equiv g(a)\equiv0\bmod p$. The order of both $f(x)$ and $g(x)$ is thus at least one and we can write
\begin{align*}
&\int_{a+(p\Zp)^n}|f(x)|^s|g(x)||dx|\\&=\sum_{k=1}^{\infty}\sum_{l=1}^{\infty}\int_{\substack{x\in a+(p\Zp)^2\\\ord f(x)=k,\ \ord g(x)=l}}|f(x)|^s|g(x)||dx|\\
&=\sum_{k=1}^{\infty}\sum_{l=1}^{\infty}p^{-ks}p^{-l}\cdot\mu\left(\{x\in a+(p\Zp)^2\mid\ord f(x)=k\ \textrm{and}\ \ord g(x)=l\}\right).
\end{align*}
We will prove below that the measure of
\begin{equation}\label{verzameling}
\{x\in a+(p\Zp)^2\mid\ord f(x)=k\ \textrm{and}\ \ord g(x)=l\}
\end{equation}
equals $p^{-n-k-l}(p-1)^2$. Assuming it now, we can continue:
\begin{align*}
\int_{a+(p\Zp)^n}|f(x)|^s|g(x)||dx|&=\sum_{k=1}^{\infty}\sum_{l=1}^{\infty}p^{-ks-l}p^{-n-k-l}(p-1)^2\\
&=p^{-n}(p-1)^2\sum_{k=1}^{\infty}p^{-(s+1)k}\sum_{l=1}^{\infty}p^{-2l}\\
&=p^{-n}(p-1)^2\frac{p^{-(s+1)}}{1-p^{-(s+1)}}\,\frac{p^{-2}}{1-p^{-2}}\\
&=p^{-n}\frac{p-1}{(p^{s+1}-1)(p+1)},
\end{align*}
giving the result.

It remains to prove that the measure of \eqref{verzameling} is equal to $p^{-n-k-l}(p-1)^2$. With the notation of Lemma~\ref{notAkl}, we have that $A_{k,l}\supset A_{k+1,l}$ and $A_{k,l}\supset A_{k,l+1}$ and that the set $\{x\in a+(p\Zp)^2\mid\ord f(x)=k\ \textrm{and}\ \ord g(x)=l\}$ is the complement of $A_{k+1,l}\cup A_{k,l+1}$ in $A_{k,l}$. Noting further that $A_{k+1,l}\cap A_{k,l+1}=A_{k+1,l+1}$ and making use of Lemma~\ref{notAkl}, we can calculate that
\begin{align*}
&\mu\left(\{x\in a+(p\Zp)^2\mid\ord f(x)=k\ \textrm{and}\ \ord g(x)=l\}\right)\\
&=\mu(A_{k,l})-\mu(A_{k+1,l})-\mu(A_{k,l+1})+\mu(A_{k+1,l+1})\\
&=p^{-n-k-l+2}-2p^{-n-k-l+1}+p^{-n-k-l}\\
&=p^{-n-k-l}(p-1)^2.
\end{align*}
\end{proof}

\begin{corollary}\label{corol3}
Let $n\geqslant2$ and $f,g\in\Zp[x_1,\ldots,x_n]$. Put
\begin{align*}
N&=\#\{a\in\Fpcrossn\mid\overline{f}(a)=0\textrm{ and }\overline{g}(a)\neq0\},\\
P&=\#\{a\in\Fpcrossn\mid\overline{f}(a)\neq0\textrm{ and }\overline{g}(a)=0\},\\
Q&=\#\{a\in\Fpcrossn\mid\overline{f}(a)=\overline{g}(a)=0\}.
\end{align*}
Suppose that none of the sets of congruence relations \eqref{stelsels} from the previous proposition has a solution in \Zpxn. Then for $s\in\C$ with $\Re(s)>0$, we have
\begin{multline*}
\int_{\Zpxn}|f(x)|^s|g(x)||dx|=\\p^{-n}\left((p-1)^n-pN\frac{p^s-1}{p^{s+1}-1}-P\frac{p}{p+1}-pQ\frac{p^s(p+1)-2}{(p^{s+1}-1)(p+1)}\right).
\end{multline*}
\end{corollary}

\begin{proof}
Splitting up the integration domain as follows, we get
\begin{multline*}
\int_{\Zpxn}|f(x)|^s|g(x)||dx|=\\
\begin{aligned}
\sum_{\substack{a\in\{1,\ldots,p-1\}^n\\f(a)\not\equiv0\bmod p\\g(a)\not\equiv0\bmod p}}\int_{a+(p\Zp)^n}|f(x)|^s|g(x)||dx|&+\sum_{\substack{a\in\{1,\ldots,p-1\}^n\\f(a)\equiv0\bmod p\\g(a)\not\equiv0\bmod p}}\int_{a+(p\Zp)^n}|f(x)|^s|g(x)||dx|\\
+\sum_{\substack{a\in\{1,\ldots,p-1\}^n\\f(a)\not\equiv0\bmod p\\g(a)\equiv0\bmod p}}\int_{a+(p\Zp)^n}|f(x)|^s|g(x)||dx|&+\sum_{\substack{a\in\{1,\ldots,p-1\}^n\\f(a)\equiv0\bmod p\\g(a)\equiv0\bmod p}}\int_{a+(p\Zp)^n}|f(x)|^s|g(x)||dx|.
\end{aligned}
\end{multline*}
Applying Proposition~\ref{prop}, we find
\begin{multline*}
\int_{\Zpxn}|f(x)|^s|g(x)||dx|\\
\begin{aligned}
&=((p-1)^n-N-P-Q)p^{-n}+Np^{-n}\frac{p-1}{p^{s+1}-1}+Pp^{-n}\frac{1}{p+1}\\
&\quad +Qp^{-n}\frac{p-1}{(p^{s+1}-1)(p+1)}\\
&=p^{-n}\left((p-1)^n-pN\frac{p^s-1}{p^{s+1}-1}-P\frac{p}{p+1}-pQ\frac{p^s(p+1)-2}{(p^{s+1}-1)(p+1)}\right).
\end{aligned}
\end{multline*}
\end{proof}

\section{Igusa's \except{toc}{local }zeta function of a single polynomial and a polynomial measure}\label{formule2}
Let $f,g$ be non-zero polynomials in $\Zp[x_1,\ldots,x_n]$ without a constant term. Suppose $f$ and $g$ are both non-degenerated\footnote{See Definition~\ref{def_non-degenerated}.} over \Fp\ with respect to all the faces of their Newton polyhedron. We will consider the Newton polyhedra \Gf\ and \Gg\ of $f$ and $g$, and their corresponding partitions\footnote{See Definitions~\ref{def_Df} and \ref{def_Dfg}.} \Df, \Dg, and \Dfg\ of \Rplusn.

\begin{notation}
As before, each cone $\delta$ in \Dfg\ can be written in a unique way as the intersection $\delta=\Dft\cap\Dg(\tau')$ of a cone \Dft\ in \Df\ with a cone $\Dg(\tau')$ in \Dg. So to each cone $\delta\in\Dfg$ we can associate a face $\tau$ of \Gf\ and a face $\tau'$ of \Gg. For $\delta\in\Dfg$, denote by \fd\ the polynomial\footnote{See Definition~\ref{def_non-degenerated}.} \ft, with $\tau$ the face of \Gf\ associated to $\delta$. The polynomial \gd\ is defined in the same way. Furthermore, for $\delta\in\Dfg$, we put
\begin{align*}
\Nd&=\#\{a\in\Fpcrossn\mid\fbard(a)=0\textrm{ and }\gbard(a)\neq0\},\\
\Pd&=\#\{a\in\Fpcrossn\mid\fbard(a)\neq0\textrm{ and }\gbard(a)=0\},\\
\Qd&=\#\{a\in\Fpcrossn\mid\fbard(a)=\gbard(a)=0\}.
\end{align*}
\end{notation}

\begin{definition}[Non-degenerated]\label{non-degeneratedII}
Let $n\geqslant2$ and $f,g$ be non-zero polynomials in $\Zp[x_1,$ $\ldots,x_n]$ without a constant term. We say that the pair $(f,g)$ is non-degenerated over \Fp\ with respect to all the cones in the partition \Dfg\ of \Rplusn, associated to $f$ and $g$, if for every cone $\delta\in\Dfg$ and all $a\in\Zpxn$, such that $\fd(a)\equiv\gd(a)\equiv0\bmod p$, the Jacobian matrix $J(\fd,\gd,a)$ has rank $2$ modulo $p$. This amounts to saying that for every cone $\delta\in\Dfg$, the set of congruences
\begin{equation*}
\left\{
\begin{aligned}
\fd(x)&\equiv0\bmod p,\\
\gd(x)&\equiv0\bmod p,\\
\begin{vmatrix}
\frac{\partial \fd}{\partial x_i}(a)&\frac{\partial \fd}{\partial x_j}(a)\\
\frac{\partial \gd}{\partial x_i}(a)&\frac{\partial \gd}{\partial x_j}(a)
\end{vmatrix}
&\equiv0\bmod p;\quad i,j=1,\ldots,n;\ i<j;
\end{aligned}
\right.
\end{equation*}
has no solutions in \Zpxn.
\end{definition}

%
\begin{theorem}\label{formule_Zfgs}
Let $p$ be a prime number and $n\geqslant2$. Let $f,g$ be non-zero polynomials in $\Zp[x_1,\ldots,x_n]$ with $f(0)=0$ and $g(0)=0$. Suppose that $f$ and $g$ are both non-degenerated over \Fp\ with respect to all the faces of their Newton polyhedron and that $(f,g)$ is non-degenerated over \Fp\ with respect to all the cones in the partition \Dfg\ of \Rplusn, associated to $f$ and $g$. Then we have:
\begin{equation*}
\Zfgs=\sum_{\delta\in\Dfg}\Ld\Sd,
\end{equation*}
with
\begin{equation*}
\Ld=p^{-n}\left((p-1)^n-p\Nd\frac{p^s-1}{p^{s+1}-1}-\Pd\frac{p}{p+1}-p\Qd\frac{p^s(p+1)-2}{(p^{s+1}-1)(p+1)}\right)
\end{equation*}
and
\begin{equation*}
\Sd=\sum_{k\in\Nn\cap\delta}p^{-m_f(k)s-m_g(k)-\sigma(k)},
\end{equation*}
for every cone $\delta$ in \Dfg.

Just as in Theorem~\ref{theformule1}, the \Sd\ can be calculated by considering a partition of $\delta$ into rational simplicial cones $\delta_i$, $i\in I$, without the introduction of new rays. If for each $i$, we take $\delta_i$ to be strictly positively spanned by the linearly independent, primitive vectors $k_{i,1},\ldots,k_{i,r_i}\in\Nn\setminus\{0\}$, then we have
\begin{equation*}
\Sd=\sum_{i\in I}\frac{\sum_hp^{m_f(h)s+m_g(h)+\sigma(h)}}{\prod_{j=1}^{r_i}(p^{m_f(k_{i,j})s+m_g(k_{i,j})+\sigma(k_{i,j})}-1)},
\end{equation*}
where $h$ runs through the elements of the set
\begin{equation*}
\Zn\cap\left\{\sum\nolimits_{j=1}^{r_i}\lambda_jk_{i,j}\;\middle\vert\;0\leqslant\lambda_j<1\text{ for }j=1,\ldots,r_i\right\}.
\end{equation*}

The real parts of the candidate poles for \Zfg\ are therefore given by the rational numbers $-1$ and
\begin{equation*}
-\frac{m_g(k)+\sigma(k)}{m_f(k)}
\end{equation*}
for $k$ a primitive generator of a ray in \Dfg.
\end{theorem}

\begin{proof}
The proof is again similar to the proof of \cite[Theorem 4.2]{DH01}. Based on the partition \Dfg\ of \Rplusn, we split up the integration domain as follows:
\begin{align*}
\Zfgs&=\int_{\Zpn}|f(x)|^s|g(x)||dx|\\
&=\sum_{\delta\in\Dfg}\sum_{k\in\Nn\cap\delta}\int_{\substack{x\in\Zpn\\\ord x=k}}|f(x)|^s|g(x)||dx|.
\end{align*}

To calculate the integral over $\{x\in\Zpn\mid\ord x=k\}$, we will make a change of variables. For $\delta\in\Dfg$, $k\in\Nn\cap\delta$, and $x\in\Zpn$ with $\ord x=k$, put $x_j=p^{k_j}u_j$ with $u_j\in\Zpx$. Then
\begin{align*}
|dx|&=p^{-\sigma(k)}|du|\qquad\textrm{and}\\
x^{\omega}&=x_1^{\omega_1}\cdots x_n^{\omega_n}=p^{k\cdot\omega}u^{\omega}.
\end{align*}
Note that $f(x)$ is a \Zp-linear combination of the $x^{\omega}$ with $\omega\in\supp(f)$, and thus of the $p^{k\cdot\omega}u^{\omega}$ in the new variables. For $k$ fixed and $\omega\in\supp(f)$, the scalar product $k\cdot\omega$ is minimal with value $m_f(k)$ for $\omega\in F_f(k)=\tau$, where $\tau$ is the face of \Gf\ associated to $\delta$, and not minimal outside $\tau$ (Cfr.\ Definitions~\ref{def_mf}, \ref{def_firstmeetlocus}, and \ref{def_Df}). Separating the maximal power of $p$, we can write
\begin{equation*}
f(x)=p^{m_f(k)}(\ft(u)+p\tilde{f}_{\tau,k}(u))=p^{m_f(k)}(\fd(u)+p\tilde{f}_{\tau,k}(u)),
\end{equation*}
with $\tilde{f}_{\tau,k}(u)$ a polynomial in $\Zp[u_1,\ldots,u_n]$, depending on $f$, $\tau$, and $k$. Analogously, we can write $g(x)$ as
\begin{equation*}
g(x)=p^{m_g(k)}(g_{\tau'}(u)+p\tilde{g}_{\tau',k}(u))=p^{m_g(k)}(\gd(u)+p\tilde{g}_{\tau',k}(u)).
\end{equation*}
This all leads to $\Zfgs=$
\begin{equation*}
\sum_{\delta\in\Dfg}\sum_{k\in\Nn\cap\delta}p^{-m_f(k)s-m_g(k)-\sigma(k)}\int_{\Zpxn}|\fd(u)+p\tilde{f}_{\tau,k}(u)|^s|\gd(u)+p\tilde{g}_{\tau',k}(u)||du|.
\end{equation*}

Let us put
\begin{equation*}
\Ld=\int_{\Zpxn}|\fd(u)+p\tilde{f}_{\tau,k}(u)|^s|\gd(u)+p\tilde{g}_{\tau',k}(u)||du|.
\end{equation*}
Because of the non-degeneracy conditions on $f$ and $g$, for every $\delta\in\Dfg$ the polynomials \fd\ and \gd\ satisfy the condition formulated in the statement of Corollary~\ref{corol3}. It follows immediately that the polynomials $\fd(u)+p\tilde{f}_{\tau,k}(u)$ and $\gd(u)+p\tilde{g}_{\tau',k}(u)$ satisfy the same condition. So by Corollary~\ref{corol3}, we know that
\begin{align*}
\Ld&=\int_{\Zpxn}|\fd(u)+p\tilde{f}_{\tau,k}(u)|^s|\gd(u)+p\tilde{g}_{\tau',k}(u)||du|\\
&=p^{-n}\left((p-1)^n-p\Nd\frac{p^s-1}{p^{s+1}-1}-\Pd\frac{p}{p+1}-p\Qd\frac{p^s(p+1)-2}{(p^{s+1}-1)(p+1)}\right).
\end{align*}
Because this expression for \Ld\ is independent of $k$, we can write
\begin{equation*}
\Zfgs=\sum_{\delta\in\Dfg}\Ld\sum_{k\in\Nn\cap\delta}p^{-m_f(k)s-m_g(k)-\sigma(k)}=\sum_{\delta\in\Dfg}\Ld\Sd,
\end{equation*}
giving the main formula. The \lq closed\rq\ formula for \Sd, based on simplicial subdivision of $\delta$, can be found following the argumentation of \cite[Theorem~4.2]{DH01}.
\end{proof}

\section{Igusa's \except{toc}{local }zeta function of a polynomial mapping and a polynomial measure}\label{formule3}
Let $f_1,\ldots,f_t,g\in\Zp[x_1,\ldots,x_n]$, such that $\ff=(f_1,\ldots,f_t):\Qpn\to\Qp^t$ is a nonconstant polynomial mapping, satisfying $\ff(0)=0$. In this section, we give an explicit formula for Igusa's local zeta function \Zffg, associated to \ff\ and the integration measure $|g(x)||dx|$ on \Zpn, in the same style as the previous ones. This time we adapt a formula for Igusa's local zeta function of a polynomial mapping, given in \cite{VZ08}, which is a generalization of the formulas proven in \cite{DH01} and \cite{HMY07}. In the same way the formula we will state here, generalizes the main results of Sections~\ref{formule1} and \ref{formule2}. Again, we restrict to the case where \ff\ and $g$ satisfy a number of non-degeneracy conditions (see below).

Because the derivation of the formula is completely analogous to what we have done in Sections~\ref{sectHensel} and \ref{formule2} for a single polynomial, in this section, we will only give the results and omit the proofs. We start by stating the analogues of some results from Section~\ref{sectHensel}, to conclude with the formula itself in Theorem~\ref{theo_form3}.

\begin{lemma}\label{notAkl_ff}
Let \ff\ and $g$ be as above and suppose that $n\geqslant t+1$. Let $k,l\in\Nnul$ and $a=(a_1,\ldots,a_n)\in\Zp^n$, such that $\ff(a)\equiv0\bmod p$ and $g(a)\equiv0\bmod p$, and such that the Jacobian matrix $J(\ff,g,a)=J(f_1,\ldots,f_t,g,a)$ has rank $t+1$ modulo $p$. Then the measure of the set
\begin{equation*}
A_{k,l}=\{x=(x_1,\ldots,x_n)\in a+(p\Zp)^n\mid \ff(x)\equiv0\bmod p^k\ \textrm{\normalfont and}\ g(x)\equiv0\bmod p^l\}
\end{equation*}
equals $p^{-n-(k-1)t-l+1}$.
\end{lemma}

\begin{proposition}\label{prop_ff}
Let \ff, $g$, and $n$ be as above. Let $a\in\Zp^n$, such that the following three conditions hold: if $a$ is in the zero locus of $\ff,g,(\ff,g)$, respectively, then the Jacobian matrix of $\ff,g,(\ff,g)$, respectively, and $a$, has maximal rank ($\star$). Then for $s\in\C$ with $\Re(s)>0$, we have
\begin{equation*}
\int_{a+(p\Zp)^n}\|\ff(x)\|^s|g(x)||dx|=
\begin{cases}
p^{-n},&\textrm{\normalfont if $\ff(a)\not\equiv0$ and $g(a)\not\equiv0\bmod p$;}\\
p^{-n}\frac{p^t-1}{p^{s+t}-1},&\textrm{\normalfont if $\ff(a)\equiv0$ and $g(a)\not\equiv0\bmod p$;}\\
p^{-n}\frac{1}{p+1},&\textrm{\normalfont if $\ff(a)\not\equiv0$ and $g(a)\equiv0\bmod p$;}\\
p^{-n}\frac{p^t-1}{(p^{s+t}-1)(p+1)},&\textrm{\normalfont if $\ff(a)\equiv0$ and $g(a)\equiv0\bmod p$.}
\end{cases}
\end{equation*}
\end{proposition}

\begin{corollary}\label{corol3_ff}
Let \ff, $g$, and $n$ be as above. Put
\begin{align*}
N&=\#\{a\in\Fpcrossn\mid\overline{\ff}(a)=0\textrm{ and }\overline{g}(a)\neq0\},\\
P&=\#\{a\in\Fpcrossn\mid\overline{\ff}(a)\neq0\textrm{ and }\overline{g}(a)=0\},\\
Q&=\#\{a\in\Fpcrossn\mid\overline{\ff}(a)=0\textrm{ and }\overline{g}(a)=0\}.
\end{align*}
Suppose that the three conditions ($\star$) from the previous proposition hold for all $a\in\Zpxn$. Then for $s\in\C$ with $\Re(s)>0$, we have
\begin{multline*}
\int_{\Zpxn}\|\ff(x)\|^s|g(x)||dx|=\\p^{-n}\left((p-1)^n-p^tN\frac{p^s-1}{p^{s+t}-1}-P\frac{p}{p+1}-pQ\frac{p^{t-1}(p^s(p+1)-1)-1}{(p^{s+t}-1)(p+1)}\right).
\end{multline*}
\end{corollary}

\begin{notation}
Again each cone $\delta$ in \Dffg\ can be written in a unique way as the intersection $\delta=\Dfft\cap\Dg(\tau')$ of a cone \Dfft\ in \Dff\ with a cone $\Dg(\tau')$ in \Dg. So to each cone $\delta\in\Dffg$ we can associate a face $\tau$ of \Gff\ and a face $\tau'$ of \Gg. For $\delta\in\Dffg$, denote by \ffd\ the polynomial \fft, with $\tau$ the face of \Gff\ associated to $\delta$, and analogously for \gd. Furthermore, for $\delta\in\Dffg$, we put
\begin{align*}
\Nd&=\#\{a\in\Fpcrossn\mid\ffbard(a)=0\textrm{ and }\gbard(a)\neq0\},\\
\Pd&=\#\{a\in\Fpcrossn\mid\ffbard(a)\neq0\textrm{ and }\gbard(a)=0\},\\
\Qd&=\#\{a\in\Fpcrossn\mid\ffbard(a)=0\textrm{ and }\gbard(a)=0\}.
\end{align*}
\end{notation}

We need the following notion of non-degeneracy.

\begin{definition}[strongly non-degenerated]\label{non-degeneratedIII}
Let \ff, $g$, and $n$ be as above. We say that the pair $(\ff,g)$ is strongly non-degenerated over \Fp\ with respect to (all the cones of) the partition \Dffg\ of \Rplusn, associated to \ff\ and $g$, if for every cone $\delta\in\Dffg$ and all $\overline{a}\in\Fpcrossn$, such that $\overline{f_{1,\delta}}(\overline{a})=\cdots=\overline{f_{t,\delta}}(\overline{a})=\gbard(\overline{a})=0$, the Jacobian matrix
\begin{equation*}
J(\overline{\ffd},\gbard,\overline{a})=
\begin{pmatrix}
(\partial \overline{f_{1,\delta}}/\partial x_1)(\overline{a})&\cdots&(\partial \overline{f_{1,\delta}}/\partial x_n)(\overline{a})\\
\vdots&\ddots&\vdots\\
(\partial \overline{f_{t,\delta}}/\partial x_1)(\overline{a})&\cdots&(\partial \overline{f_{t,\delta}}/\partial x_n)(\overline{a})\\[+1ex]
(\partial \gbard/\partial x_1)(\overline{a})&\cdots&(\partial \gbard/\partial x_n)(\overline{a})
\end{pmatrix}
\end{equation*}
has maximal rank ($=t+1$).
\end{definition}

\begin{theorem}\label{theo_form3}
Let $p$ be a prime number. Let $\ff=(f_1,\ldots,f_t):\Qpn\to\Qp^t$ be a nonconstant polynomial mapping, with $f_1,\ldots,f_t\in\Zp[x_1,\ldots,x_n]$, satisfying $f_i(0)=0$ for all $i$, and suppose that $n\geqslant t+1$. Let $g$ be a nonconstant polynomial in $\Zp[x_1,\ldots,x_n]$ with $g(0)=0$, and suppose that $g$ is non-degenerated over \Fp\ with respect to its Newton polyhedron. Suppose also that \ff\ and $(\ff,g)$ are strongly non-degenerated over \Fp\ with respect to the Newton polyhedron \Gff\ and the partition \Dffg\ of \Rplusn, respectively. Then we have:
\begin{equation*}
\Zffgs=\sum_{\delta\in\Dffg}\Ld\Sd,
\end{equation*}
with
\begin{equation*}
\Ld=p^{-n}\left((p-1)^n-p^t\Nd\frac{p^s-1}{p^{s+t}-1}-\Pd\frac{p}{p+1}-p\Qd\frac{p^{t-1}(p^s(p+1)-1)-1}{(p^{s+t}-1)(p+1)}\right)
\end{equation*}
and
\begin{equation*}
\Sd=\sum_{k\in\Nn\cap\delta}p^{-m_{\ff}(k)s-m_g(k)-\sigma(k)},
\end{equation*}
for every cone $\delta$ in \Dffg.

Again\footnote{Cfr.\ Theorems~\ref{theformule1} and \ref{formule_Zfgs}.} there exists a closed, rational expression for the \Sd, showing that the real parts of the candidate poles for \Zffg\ are given by the rational numbers $-1$ and
\begin{equation*}
-\frac{m_g(k)+\sigma(k)}{m_{\ff}(k)}
\end{equation*}
for $k$ a primitive generator of a ray in \Dffg.
\end{theorem}

\begin{proof}
The proof of the theorem is completely similar to the proof of Theorem~\ref{formule_Zfgs}.
\end{proof}

\bibliography{bartbib}
\bibliographystyle{amsplain}
\end{document}